\documentclass[11pt]{amsart}
\usepackage{amsmath,amssymb,amsbsy,amsfonts,amsthm,latexsym,amsopn,amstext,amsxtra,epic, euscript,amscd,indentfirst,}
\usepackage[margin=1.2in]{geometry}
\begin{document}

\newtheorem{theorem}{Theorem}
\newtheorem*{theorem*}{Theorem}
\newtheorem{conjecture}[theorem]{Conjecture}
\newtheorem*{conjecture*}{Conjecture}
\newtheorem{proposition}[theorem]{Proposition}
\newtheorem{question}[theorem]{Question}
\newtheorem{lemma}[theorem]{Lemma}
\newtheorem{cor}[theorem]{Corollary}
\newtheorem{obs}[theorem]{Observation}
\newtheorem{proc}[theorem]{Procedure}
\newcommand{\comments}[1]{} 
%% DEFINITIONS
\def\Z{\mathbb Z}
\def\Za{\mathbb Z^\ast}
\def\Fq{{\mathbb F}_q}
\def\R{\mathbb R}
\def\N{\mathbb N}
\def\C{\mathbb C}
\def\k{\kappa}
\def\grad{\nabla}

\title[The Drift Laplacian and Hermitian Geometry]{Eigenvalues of the Complex Laplacian on compact non-K\"ahler manifolds}

\author{Gabriel J. H. Khan}
\email{khan.375@osu.edu}

\date{\today}

\maketitle

\begin{abstract}
Let $(M^n, h)$ be a compact Hermitian manifold. Suppose $\lambda$ is the lowest eigenvalue of the complex Laplacian on $M$. We prove that $\lambda \geq C$ where $C$ depends only on the dimension $n$, the diameter $d$, the Ricci curvature of the Levi-Civita connection on $M$, and a norm, expressed in curvature, that determines how much $M$ fails to be K\"ahler. We first estimate the principal eigenvalue of a drift Laplacian and then study the structure of Hermitian manifolds using recent results due to Yang and Zheng~\cite{YZ}. We combine these results to obtain the main estimate.
\end{abstract}

\section{Introduction}

This preprint's main goal is to obtain a lower bound on the spectrum of the complex Laplacian on a compact Hermitian manifold. To do this, we need two seemingly unrelated results. First, we derive an estimate for the principal eigenvalue of a Laplacian with drift. Second, using recent results from ~\cite{YZ}, we find inequalities which allow us to to estimate the torsion of a Hermitian manifold in terms of the Riemannian and Hermitian curvature. We then note that the complex Laplacian can be viewed as a drift Laplacian in which the drift can be expressed in terms of the torsion. We thus get the desired estimate.

In Section 1, we state and discuss the results. In Section 2, we prove the estimate for the drift Laplacian. In Section 3, we discuss the torsion in greater depth and prove the lemmas in Section 1 that we need to bound the torsion. In the final section, we prove the estimate on the complex Laplacian and discuss some conjectures on the relation between curvature tensors, the torsion tensor, and orthogonal complex structures.

\subsection{The Drift Laplacian}

  The drift Laplacian is a natural operator that appears in physical applications. The associated heat equation has been studied and the drift term acts as convection (i.e. stirring). Often, though not always, stirring speeds up the diffusion process. Therefore, we might expect to be able to derive lower bounds on the spectrum for the drift Laplacian. Theorem 1 provides an example of such bounds.

\begin{theorem}
Let $(M^n, g)$ be a compact Riemannian manifold without boundary, satisfying $Ric~ M \geq -(n-1)k$ and $\xi$ be a one-form on $M$. Suppose $u \in C^\infty(M)$ is a solution to the equation $\Delta u + \xi(\nabla u) = \lambda u$.
 Let $\| \xi \|$ be the $C_0$ norm of $\xi$ and $\| \nabla \xi \|$ be the $C_0$ norm of $\nabla \xi$ (as a two tensor). Let 
$D =  2nd^2$ where $d$ is the diameter of M and $E = \frac{1}{2n}( (16n^2-32n+5) \| \xi \|^2 + 2(n-1)^2 k + 2(n-1) \|\nabla \xi\|)$
Then we have: $$\lambda \geq \frac{1}{D}  \frac{(1 + \sqrt{1+4DE})^2-DE}{\exp(1 + \sqrt{1+4DE})}$$
\end{theorem}

 Similar (and sharper) estimates have been obtained in the case of the Witten-Laplacian, where $\xi = df$ for some smooth function $f$, such as in ~\cite{AN} ~\cite{FLL}. For our purposes, $\xi$ will generally not be exact so we cannot use these results. The one-form is exact if and only if the metric is conformal to a balanced metric, which is a very restrictive condition.
  We would not be surprised if Theorem 1 were already known but we have not been able to find it in our literature search. In \cite{GN}, Gonzalez and Negrin study the kernel of the drift Laplacian on open domains with the same conditions that we use. More recently, Jorgen Jost and others have studied harmonic maps for a generalization of this operator (so called V-Harmonic maps). This group has proven various results and advanced the theory of harmonic maps ~\cite{CJQ}.

\subsection{Structural inequalities on Hermitian manifolds}
 In Section 3, we use recent results from ~\cite{YZ} to derive inequalities that estimate the torsion of a Hermitian manifold in terms of the Riemannian and Hermitian curvature. We define two norms that measure the difference between Hermitian and Riemannian curvature, denoted $R^h$ and $R$, respectively.

  Given a unitary frame $\{e_i\}$ on a Hermitian manifold, we define  $\| R^h-R\|^2_\star$ and  $\| R^h-R \|^2_{\star \star}$ in the following way:
  $$\| R^h-R \|^2_\star = \sum_{i,j,k,l}| R^h_{i \bar j k \bar l} - R_{i \bar j k \bar l} |^2 + 2\sum_{i,j,k,l}| R_{i j \bar k \bar l} |^2 $$
 $$\| R^h-R \|^2_{\star \star} = \sum_{i,j,k,l} | R_{i j k \bar l} |^2 $$

 We show that these quantities dominate the $C^0$ and $C^1$ norm of the torsion. To be precise, let $\nabla^{c'}$ and $\nabla^{c''}$ be the $(1,0)$ and $(0,1)$ components of the covariant differentiation of the Chern connection defined by:

$$\nabla^{c'}_{X+\overline{Y}} T = \nabla^c_X T \textrm{ and }$$
$$\nabla^{c''}_{X+\overline{Y}} T = \nabla^c_{\overline{Y}} T $$

where $X$ and $Y$ are any complex tangent vectors on M of type $(1,0).$
  
  \begin{theorem}
 The following inequalities hold pointwise:
$$||T||^2 \leq \| R^h-R \|_\star \textrm{ and } ||\eta||^2 \leq  \| R^h-R \|_\star$$
$$\| \nabla^{c'} (T)\| \leq  \| R^h-R \|_\star$$
$$\| \nabla^{c''} T\| \leq C(n) \| R^h-R \|_\star + \| R^h-R \|_{\star \star}$$
\end{theorem}

Here, $\eta$ is Gauduchon's torsion one-form, defined by $\partial \omega^{n-1} = -2 \eta \wedge \omega^{n-1}$, where $\omega$ is the K\"ahler (metric) form of the metric h. Given a unitary frame $\{e_i\}$, we can also define $\eta$ as $\eta_i = \sum_j T_{ij}^j$.
 
 The torsion expresses the difference between the Levi-Civita connection and the Hermitian connection. Therefore, given a unit vector $X$ of type (1,0), the difference between $\nabla^{c''}_X T$ and $ \nabla_{\bar X} T$ can be bounded by a quadratic expression in torsion. We can bound the difference between $ \nabla^{c'}_X T$ and $ \nabla_X T$ in the same way. Using these observations, we obtain Theorem 3.
 
 \begin{theorem}
Let $T$ be the torsion tensor and $\nabla T$ the derivative of the torsion tensor with respect to the Levi-Civita connection. Then there exists $C^\prime(n)$ so that following inequality holds:
$$||\nabla T|| \leq   C^\prime(n) \| R^h-R \|_\star + \| R^h -R \|_{\star \star} $$
\end{theorem}
 
 \subsection{The Complex Laplacian}
 
  Using the observation that the complex Laplacian on a Hermitian manifold can be expressed as a Laplacian with drift, we translate our estimate on the eigenvalue on the drift Laplacian into an estimate on the Laplacian on a Hermitian manifold. 
  
\begin{theorem}
Suppose that $(M^n, h)$ is a compact, Hermitian manifold. Then there exists a uniform $C >0$ such that:
$$\lambda \geq \frac{1}{4n}  \frac{\left( \frac{2}{d^2} + 3Cn^2 (k+ \| R-R^h \|_\star + \| R-R^h \|_{\star \star} ) \right)}{\exp \left( 1 + \sqrt{1+4 Cn^2 d^2 (k+\| R-R^h \|_\star + \| R-R^h \|_{\star \star}) } \right) }$$
\end{theorem}

This estimate is unsightly, but only involves the dimension, the diameter, the Ricci curvature, and the norms we defined earlier. Furthermore, the estimate scales as expected. Spectral geometry of Hermitian manifolds has been studied \cite{JP} \cite{PG}, especially in the context of finding spectral conditions which ensure a Hermitian manifold is balanced or K\"ahler. Our results suggest that one can understand the spectral geometry of Hermitian manifolds by studying the torsion. In a future preprint, we will try to strengthen these estimates and prove other results in this vein. We put forth the following conjecture that this estimate can be improved to only involve the Riemannian curvature tensor.

\begin{conjecture}{}
 Given a compact Hermitian manifold $(M^n, h)$, there exists $C$ depending only on the Riemannian geometry such that if $\square u = \lambda u$ then $\lambda \geq C$.
 \end{conjecture}
  This would mirror the case for the Laplace-Beltrami operator, where an estimate exists in terms of the dimension, diameter, and Ricci curvature. In order to do this, one generally tries to obtain some estimate on the torsion one-form.
We can show that there are certain curvature conditions which force $\eta$ to vanish, but we have not been able to establish this more generally. However, we can prove Conjecture 5 in several special cases.

\begin{theorem*}
Let $(M^n, g)$ be a compact globally conformally flat Hermitian manifold. Let $K= \inf_{x \in M} Ric~ M$, $k= \sup_{x \in M} Ric~ M$ $R$ be the scalar curvature of $M$, $d$ be the diameter of $M$ and $i$ be the injectivity radius of $M$. If $\lambda_1$ is the principle eigenvalue of the complex Laplacian $\square$, then we have the following estimate:
$$\lambda \geq C(d, K, k, n, | \nabla R|, R^2, i)$$
\end{theorem*}

 \begin{theorem}Let $(M^{2n}, g)$ be a compact Riemannian manifold and $J$ be an orthogonal complex structure which is $k$-Gauduchon for some $k>\frac{n}{2}$. Then the spectrum of the complex Laplacian is bounded below by some constant $C$ depending only on $(M^{2n}, g)$, independent of $J$.
 \end{theorem}
 
 The proof of Theorem 6 is by contradiction, and so we are not able to produce an numeric lower bound in this paper in terms of the geometry of $M$. We will do so in a future preprint.

%One possible way to attempt to prove Conjecture 5 would be to prove a Faber-Krahn inequality. The Laplacian with drift has been studied on smoothly-bounded domains in $\R^n$ and in this setting, the following paper proves a version of the Faber-Krahn inequality that only relies on the $C_0$ norm of the drift ~\cite{HNR}. If some version of this result were true on compact manifolds, one could use this result to study Conjecture 5.

For future work, we hope to continue studying torsion and to try to understand the moduli space of complex structures which are orthogonal to a given Riemannian metric. This would show how much information the Riemannian geometry can detect about the complex structure. There are strong restrictions preventing a metric from having a compatible a complex structure. To give a result in this vein, Gauduchon proved that hyperbolic manifolds of dimension great than two do not admit complex structures, ~\cite{PG}, a result that Hernandez-lamoneda extended to hold for negative strictly quarter-pinched manifolds ~\cite{LH}. However, one would hope interesting results also hold when the moduli space is not just the empty set. In such a case, a theorem of Salamon shows that at any point of a $4$-manifold, there is an open neighborhood admitting zero, one, two, or infinitely many orthogonal complex structures ~\cite{SMS}. It is worthwhile to note that in order for a single metric to provide a counterexample to Conjecture 5, it would need to admit infinitely many compatible complex structures. This is a very restrictive condition and it may be the case that metrics with infinitely many complex structures are well behaved enough that the torsion one form is controlled. In this case, Conjecture 5 would be true. However, this would not be a particularly satisfying result since it could still be possible that a sequence of metrics whose curvature and geometry is bounded in some very strong norm could be a blow up sequence for the torsion.

\subsection{Acknowledgements} We owe many thanks to Bo Guan, Bo Yang, Adrian Lam, and Fangyang Zheng for their insights and help deriving these results. Finally, thank you to Kori Brady and Fangyang Zheng for their edits and help making the writing more clear.

\section{Estimating the principal eigenvector of the drift laplacian}

We now prove Theorem 1, which gives an estimate for the principal eigenvalue of the drift Laplacian $L = \Delta + \xi(\nabla)$. This proof is an adaptation of the Li-Yau estimate given in Lectures in Differential Geometry ~\cite{SY} for the principal eigenvalue of the Laplacian. Recall that Theorem 1 states the following:

\begin{theorem*} Suppose $(M^n, g)$ is a compact Riemannian manifold without boundary satisfying $Ric~M \geq -(n-1)k$ for $k \geq 0$ and diameter $d$. Suppose that $u$ satisfies: $$Lu+\lambda u = 0$$
 Let $\| \xi \|$ be the $C_0$ norm of $\xi$ and $\| \nabla \xi \|$ be the $C_0$ norm of $\nabla_X \xi(X)$ for $\| X \| = 1$. If 
$D =  2nd^2$ where $d$ is the diameter of M and $E = \frac{1}{2n}( (16n^2-32n+5) \| \xi \|^2 + 2(n-1)^2 k + 2(n-1) \| \nabla \xi \|)$. 
Then we have $$\lambda \geq \frac{1}{D}  \frac{(1 + \sqrt{1+4DE})^2-DE}{\exp(1 + \sqrt{1+4DE})}$$
\end{theorem*}

\begin{proof}

The proof uses the standard technique of utilizing a Bochner identity on a family of functions $G_\beta(x)$ to establish a gradient estimate and then integrating the estimate to obtain an inequality involving $\lambda$ (as well as the other terms that appear). Finally, we solve the inequality in terms of $\lambda$ and optimize the estimate in terms of $\beta$ to get the desired result. However, the details in the calculation are somewhat messy.

Suppose $u$ satisfies:
\begin{equation}\label{eq:ADE}
\Delta u + \xi(\nabla u) + \lambda u = 0
\end{equation}
 with $\sup u = 1$. Let $\beta >1$ and consider the function $G(x)$ defined by: 
 \begin{equation}\label{eq:G} 
 G(x)= \dfrac{|\nabla u|^2}{(\beta-u)^2} 
 \end{equation} 
 
 We establish an estimate on $G(x)$ and use this to derive an estimate on $u$. 

\subsection{Set up using the Bochner formula}

Then suppose that $G(x)$ is maximized at $x_0.$ We have that $\nabla G(x_0) =0$ and $\Delta G(x_0)  \leq 0$.

Also, we have $G(x)(\beta-u)^2 = |\nabla u|^2$ so: 
$$(\Delta G) (\beta-u)^2 + 2 \nabla G \nabla (\beta-u)^2 + G \Delta (\beta-u)^2 = \Delta | \nabla u|^2$$
At $x_0$, we have $\nabla G(x_0) =0$ so using the Bochner formula in normal coordinates at $x_0$, we have the following:

\begin{eqnarray*} 
0 & \geq & \Delta | \nabla u|^2 -  G \Delta (\beta-u)^2 \\
   & = & 2 \sum_{i,j} u_{ij}^2 + 2 \sum_i u_i (\Delta u)_i + 2 Ric( \nabla u, \nabla u) - 2 G ((-\Delta u)(\beta - u) + |\nabla u|^2)
\end{eqnarray*}

Then dividing by 2, using \eqref{eq:ADE} and the curvature bounds, we have that:

\begin{eqnarray*} 
0 & \geq &  \sum_{i,j} u_{ij}^2 +  \sum_i u_i (-\xi(\nabla u) - \lambda u)_i +  Ric( \nabla u, \nabla u) -  G ((-\Delta u)(\beta - u) + |\nabla u|^2) \\
   & \geq &  \sum_{i,j} u_{ij}^2 +  \sum_i u_i (-\xi(\nabla u) - \lambda u)_i   -  (n-1)k |\nabla u|^2 -  G ((-\Delta u)(\beta - u) + |\nabla u|^2) \\
\end{eqnarray*}

We pick normal coordinates at $x_0$ so that $u_i = 0$ for $i> 1$ and $u_1 = |\nabla u|$. Then $\nabla G(x_0) =0$ implies:
\begin{equation}\label{eq:Grad 0}
u_{11} = \frac{- |\nabla u|^2}{\beta - u} \textrm{ and $u_{1i} = 0$ otherwise.}
\end{equation}

Let $\mathcal{G} = G ((\xi(\nabla u) + \lambda u)(\beta - u) + |\nabla u|^2)$. In the following manipulation we will not change this term so we use $\mathcal{G}$ as shorthand. Then, at $x_0$:
\begin{eqnarray*} 
0  & \geq &  \sum_{i,j} u_{ij}^2 +  \sum_i u_i (-\xi(\nabla u) - \lambda u)_i   -  (n-1)k |\nabla u|^2 -  G ((\xi(\nabla u) + \lambda u)(\beta - u) + |\nabla u|^2) \\
    & = &  \sum_{i,j} u_{ij}^2 +  u_1 (-\xi(\nabla u) - \lambda u)_1   -  (n-1)k |\nabla u|^2 -  \mathcal{G}  \\
    & = &  \sum_{i,j} u_{ij}^2 -  u_1 (\xi(\nabla u))_1 - \lambda u_1^2   -  (n-1)k |\nabla u|^2 -  \mathcal{G}  \\
    & = &  \sum_{i,j} u_{ij}^2 -  u_1 \xi(e_1) u_{11} - \xi(e_1)_1 u_1^2  - \lambda u_1^2   -  (n-1)k |\nabla u|^2 -  \mathcal{G}  \\
    & \geq &  \sum_{i,j} u_{ij}^2 -  u_1 |\xi| |u_{11}| - \| \nabla \xi \| u_1^2  - \lambda u_1^2   -  (n-1)k |\nabla u|^2 -  \mathcal{G} \\
    & = &  \sum_{i,j} u_{ij}^2 -  u_1 |\xi| |u_{11}| - (\| \nabla \xi \| + \lambda +  (n-1)k) |\nabla u|^2 -  \mathcal{G}  \\
\end{eqnarray*}

Therefore, we have:
\begin{equation}\label{eq:Bochner 1}
0 \geq  \sum_{i,j} u_{ij}^2 -  u_1 |\xi| |u_{11}| - (\| \nabla \xi \| + \lambda +  (n-1)k) |\nabla u|^2 -  \mathcal{G}  \\
\end{equation}

%Note that this can only be true at a point. There is no way that $(\xi(e_1))_1 = |\grad \xi|$ and |\xi| is a constant maximum along the entire curve.

\subsection{Putting the second derivatives to use}

 Continuing to work at $x_0$, we now note that:
\begin{eqnarray*} 
\sum_{i,j = 2}^n u_{ij}^2 & \geq & \sum_{i = 2}^n u_{ii}^2 \\
					& \geq & \frac{1}{n-1} \left( \sum_{i = 2}^n u_{ii} \right)^2 \\
					& = & \frac{1}{n-1} \left( \Delta u - u_{11} \right)^2 \\
					& = & \frac{1}{n-1} \left( -\xi (\nabla u) - \lambda u - u_{11} \right)^2 \\
					& = & \frac{1}{n-1} \left( \xi (u_1) + \lambda u + u_{11} \right)^2 \\
					& \geq & \frac{1}{n-1} \left( \frac{u_{11}^2}{2} - (\xi (u_1) + \lambda u)^2 \right) \\
					& \geq & \frac{1}{n-1} \left( \frac{u_{11}^2}{2} - 2 (\xi (u_1))^2 - 2 (\lambda u)^2 \right) \\
\end{eqnarray*}

Substituting this inequality into \eqref{eq:Bochner 1}, we get the following:

\begin{eqnarray*} 
0 & \geq & u_{11}^2 +  \frac{1}{n-1} \left( \frac{u_{11}^2}{2} - 2 (\xi (u_1))^2 - 2 (\lambda u)^2 \right) \\
    &    & -  u_1 |\xi| |u_{11}| - (\| \nabla \xi \| + \lambda +  (n-1)k) |\nabla u|^2 -  \mathcal{G} \\
\end{eqnarray*}
 Using \eqref{eq:G} and the definition of $\mathcal{G}$, we have:

\begin{eqnarray*} 
0 & \geq & u_{11}^2 +  \frac{1}{n-1} \left( \frac{u_{11}^2}{2} - 2 (\xi (u_1))^2 - 2 (\lambda u)^2 \right) -  u_1 |\xi| |u_{11}| \\
    &    &  - (\| \nabla \xi \| + \lambda +  (n-1)k) |\nabla u|^2 -  \frac{| \nabla u|^2}{(\beta-u)^2} ((\xi(\nabla u) + \lambda u)(\beta - u) + |\nabla u|^2) \\
\end{eqnarray*}

Then by \eqref{eq:Grad 0}, the first and last terms cancel, leaving:

\begin{eqnarray*} 
0 & \geq & \frac{1}{n-1} \left( \frac{u_{11}^2}{2} - 2 (\xi (u_1))^2 - 2 (\lambda u)^2 \right) -  u_1 |\xi| |u_{11}| \\
    &    &  - (\| \nabla \xi \| + \lambda +  (n-1)k) |\nabla u|^2 -  \frac{| \nabla u|^2}{(\beta-u)^2} (\xi(\nabla u) + \lambda u)(\beta - u) \\
    & \geq & \frac{1}{2(n-1)} u_{11}^2   -  u_1 |\xi| |u_{11}| - \left( \| \nabla \xi \| + \lambda +  (n-1)k + \frac{2}{n-1} |\xi|^2 \right) |\nabla u|^2 \\
    &    &   - \frac{| \nabla u|^2}{(\beta-u)} (\xi(\nabla u) + \lambda u) - \frac{2}{n-1} \lambda^2 u^2 \\
    & \geq & \frac{1}{2(n-1)} u_{11}^2   -  u_1 |\xi| |u_{11}| - \left( \| \nabla \xi \| + \lambda +  (n-1)k + \frac{2}{n-1} |\xi|^2 \right) |\nabla u|^2 \\
    &    &   - \frac{| \nabla u|^2}{(\beta-u)} |\xi| |\nabla u| - \lambda u \frac{| \nabla u|^2}{(\beta-u)}  - \frac{2}{n-1} \lambda^2 u^2 \\
    & = & \frac{1}{2(n-1)} \frac{| \nabla u|^4}{(\beta-u)^2}   - 2 |\xi| \frac{| \nabla u|^3}{(\beta-u)} - \left( \| \nabla \xi \| + \lambda +  (n-1)k + \frac{2}{n-1} |\xi|^2 \right) |\nabla u|^2 \\
    &    & - \lambda u \frac{| \nabla u|^2}{(\beta-u)}  - \frac{2}{n-1} \lambda^2 u^2 \\
\end{eqnarray*}

Now we divide this inequality by $(\beta-u)^2$ to obtain:

\begin{eqnarray*} 
0 & \geq & \frac{1}{2(n-1)} \frac{| \nabla u|^4}{(\beta-u)^4}   - 2 |\xi| \frac{| \nabla u|^3}{(\beta-u)^3} - \left( \| \nabla \xi \| + \lambda +  (n-1)k + \frac{2}{n-1} |\xi|^2 \right) \frac{| \nabla u|^2}{(\beta-u)^2} \\
    &    & - \lambda u \frac{| \nabla u|^2}{(\beta-u)^3}  - \frac{2}{n-1} \lambda^2 \frac{u^2}{(\beta-u)^2} \\
\end{eqnarray*}

Let $\alpha = \frac{u}{\beta - u}$ and note that $\alpha \leq \frac{1}{\beta - u} \leq \frac{1}{\beta - 1}.$

Then we can rewrite this inequality in terms of $G$ and $\alpha$ as:
\begin{equation}\label{eq:G inequality}
0 \geq \frac{1}{2(n-1)} G(x_0)^2   - 2 |\xi| G^{3/2} - \left( \| \nabla \xi \| + \lambda +  (n-1)k + \frac{2}{n-1} |\xi|^2 \right) G(x_0) - \lambda \alpha G(x_0)  - \frac{2}{n-1} \lambda^2 \alpha^2 \\
\end{equation}

Since $x_0$ maximizes $G(x_0)$, this inequality holds true for $G$ throughout $M$.

\subsection{Deriving and avoiding a quartic equation}

By \eqref{eq:G inequality}, we have
\begin{eqnarray*} 
0 & \geq & G^2   - 4(n-1) |\xi| G^{3/2} - 2(n-1) \left( \| \nabla \xi \| + \lambda +  (n-1)k + \frac{2}{n-1} |\xi|^2 + \lambda \alpha \right) G - 4 \lambda^2 \alpha^2 \\
  &= & G^2   - 4(n-1) |\xi| G^{3/2} - 2(n-1) \left( \| \nabla \xi \| + (n-1)k + \frac{2}{n-1} |\xi|^2 + \lambda ( \alpha+1) \right) G - 4 \lambda^2 \alpha^2 \\
  & \geq & G^2   - 4(n-1) |\xi| G^{3/2} - 2(n-1) \left( \| \nabla \xi \| + (n-1)k + \frac{2}{n-1} |\xi|^2 + \dfrac{\beta}{\beta -1} \lambda \right) G - 4 \lambda^2 \alpha^2 \\
\end{eqnarray*}

Letting $g = \sqrt{G}$, and $A = 4(n-1) |\xi|$, $B = 2(n-1) \left( \| \nabla \xi \| + \dfrac{\beta}{\beta -1} \lambda +  (n-1)k + \frac{2}{n-1} |\xi|^2 \right)$ and $C = 4 \lambda^2 \alpha^2$, this reduces to:

\begin{equation}\label{eq:The quartic}
0 \geq g^4 - A g^3 - B g^2 - C
\end{equation}

We could try to solve this quartic and then use that to get estimates on $\lambda$. However, it is much more straightforward to try to estimate $g$.

\begin{lemma}\label{Roots of quartics}
Given $A,B, C >0$, if x satisfies $P(x) = x^4 - Ax^3 -Bx^2 - C \leq 0$, Then
$x \leq A + \sqrt{B +\sqrt{C}} = a$.
\end{lemma}

Note that $P(x)+C = x^2 (x^2 - Ax -B)$ so it is sufficient to show that $P(a) >0$,  because increasing $x$ will only increase both factors if the latter is already positive.

Then we have the following:
\begin{eqnarray*} 
P(a) & = & (A + \sqrt{B +\sqrt{C}})^4 - A (A + \sqrt{B +\sqrt{C}})^3 - B (A + \sqrt{B +\sqrt{C}})^2 - C \\
 & = & (A + \sqrt{B +\sqrt{C}})^2 \left(  (A + \sqrt{B +\sqrt{C}})^2 - A (A + \sqrt{B +\sqrt{C}}) - B \right) - C\\
& = & (A + \sqrt{B +\sqrt{C}})^2 \left( A^2 + B + \sqrt C + 2A \sqrt{B +\sqrt{C}} - A^2 - A\sqrt{B +\sqrt{C}}) - B \right) - C \\
& = & (A + \sqrt{B +\sqrt{C}})^2 \left( \sqrt C + A \sqrt{B +\sqrt{C}} \right) - C > 0 \\
\end{eqnarray*} 

Thus the lemma is proved. However, in order to make future calculations more feasible, we note the following inequality holds:
$$A + \sqrt{B +\sqrt{C}}  \leq 2 \sqrt{ A^2 + B +\sqrt{C}}$$

Using $2 \sqrt{ A^2 + B +\sqrt{C}}$ for $g$ in our problem now, we obtain:

\begin{eqnarray*} 
g & \leq & \sqrt{ 16(n-1)^2 |\xi|^2 + 2(n-1) \left( \| \nabla \xi \| + \dfrac{\beta}{\beta -1} \lambda +  (n-1)k + \frac{2}{n-1} |\xi|^2 \right) + \sqrt{4 \lambda^2 \alpha^2}} \\
 & = & \sqrt{ (16(n-1)^2+4) |\xi|^2 +2(n-1)^2 k + 2(n-1) \| \nabla \xi \| + ( 2(n-1) \dfrac{\beta}{\beta -1} +2 \alpha) \lambda } \\
  & \leq & \sqrt{ (16n^2-32n+5) |\xi|^2 +2(n-1)^2 k + 2(n-1) \| \nabla \xi \| + 2(\dfrac{\beta (n-1)+1}{\beta -1}) \lambda } \\
    & \leq & \sqrt{ (16n^2-32n+5) |\xi|^2 +2(n-1)^2 k + 2(n-1) \| \nabla \xi \| + 2n(\dfrac{\beta}{\beta -1}) \lambda } \\
\end{eqnarray*} 

Recalling the definition of $g$, we can write this inequality as: 

\begin{equation}\label{eq:Gradient Estimate}
|\nabla u| \leq (\beta - u) \sqrt{ (16n^2-32n+5) |\xi|^2 +2(n-1)^2 k + 2(n-1) \| \nabla \xi \| + 2n(\dfrac{\beta}{\beta -1}) \lambda }
\end{equation}

\subsection{Getting an inequality on $\lambda$}

Now that we have a gradient estimate, we are most of the way done. It remains to integrate the inequality to get a $C^0$ estimate and pick $\beta$ so that this gives us a useful inequality on $\lambda$.

Take $x_1, x_2 \in M$ such that $u(x_1) =0$ and $u(x_2) =1$. Let $\gamma$ be the shortest geodesic joining $x_1$ and $x_2$. Let $d$ be the diameter of $M$. Note that the geodesics and diameter are defined in terms of the Levi-Civita connection because the length of paths depends only on the metric. We will discuss this phenomena further in future preprints.

Then:

\begin{eqnarray*} 
&& \log{\frac{\beta}{\beta-1}}~ \leq ~ \int_\gamma \frac{|\nabla u| }{\beta - u}~ \leq  \\
& \leq & d \sqrt{ (16n^2-32n+5) |\xi|^2 +2(n-1)^2 k + 2(n-1) \| \nabla \xi \| + 2n(\dfrac{\beta}{\beta -1}) \lambda }\\
\end{eqnarray*} 

That is to say:

\begin{eqnarray*} 
 \lambda \geq \frac{\beta -1}{2n \beta} \left( \frac{1}{d^2} (\log{\frac{\beta}{\beta-1}})^2 -  (16n^2-32n+5) |\xi|^2 - 2(n-1)^2 k - 2(n-1) \| \nabla \xi \|) \right) 
\end{eqnarray*}

Let $$E = \frac{1}{2n}( (16n^2-32n+5) |\xi|^2 + 2(n-1)^2 k + 2(n-1) \| \nabla \xi \|),$$ $$D = 2nd^2,$$ $$\textrm{and } x = \frac{\beta}{\beta -1} $$

Then this boils down to:  
\begin{eqnarray*} 
 \lambda \geq \frac{1}{D}  \frac{(\log {x})^2}{x} - \frac{E}{x} = f(x)
\end{eqnarray*}

\subsection{Strengthening the inequality}

Taking the derivative of this with respect to $x$, we find:
$$f^\prime(x) = \frac1{Dx^2}(\log x - (\log x)^2 + DE)$$

This is zero if $\log(x) = 1 + \sqrt{1+4DE}$, which is the value of $x$ which maximizes the right hand side. Finally, this implies that:

$$\lambda \geq \frac{1}{D}  \frac{(1 + \sqrt{1+4DE})^2-DE}{\exp(1 + \sqrt{1+4DE})} $$

\end{proof}

Despite how complicated the estimate is, notice that the quantity scales correctly under the scalar deformation $\rho g$ where $\rho$ is a positive constant.
This estimate is not optimal. There are a few places that this can be improved. We did not solve the quartic equation exactly (and if we had, inverting to solve for $\lambda$ would be unpleasant). Furthermore, the integration essentially assumes that $G(x)$ is constant. It does not effectively using a barrier function as in the theorem due to Zhong-Yang which derives optimal eigenvalue bounds for compact Riemannian manifolds with $Ric(M) \geq 0$ ~\cite{SY}. We imagine that such an estimate can be improved using the various methods of ~\cite{MC}, which we may attempt to do in the future. Also, the effect of the drift is overstated. There is no way that $(\xi(e_1))_1 = |\grad \xi|$ and $|\xi| = \xi(e_1)$ and both quantities are maximized along the entire curve that we integrate along. 

\section{The Complex Laplacian on a Hermitian Manifold}

We now apply the previous estimate to complex geometry. The complex Laplacian on a Hermitian manifold can naturally be written as a Laplacian with drift equation. If $\eta$ is the torsion one-form as before ($\eta_i$ is the $i$-th component of the form as opposed to the derivative) and $\xi$ is the Lee form, then we have the following:

\begin{eqnarray}
 \square f = \frac12 \Delta f + \xi(\nabla f) \\
 \textrm{ and } \eta + \overline{\eta} = - 2 \xi
 \end{eqnarray}

We believe that it is well known at this point, but it is worth noting that $\Delta = 2 \square$ if $\eta = 0$, which is to say that $(M, h)$ is a balanced metric. Therefore, a sufficient condition for the Laplacian and twice the complex Laplacian to be isospectral is that the metric is balanced. Another condition ensuring that the metric is balanced is that $\xi$ = 0, and $\xi$ is exact if and only if the metric is conformal to a balanced metric. Therefore, the estimates from the Witten-Laplacian can only be used in the special case where the metric is conformally equivalent to a balanced metric.

\begin{conjecture} Given, $(M^n, h)$ a complex manifold, if $\Delta$ and $2 \square$ have the same spectrum, then $(M^n, h)$ is balanced.
\end{conjecture}
 Any counterexample would be a very interesting Hermitian manifold in its own right. It is known that if the complex Laplacian on a manifold is isospectral to the complex Laplacian on a balanced manifold, then it must be balanced as well.
H. Donnelly showed in ~\cite{HD} that if $\square$ and $\Delta$ are isospectral on functions and one-forms, that the manifold is in fact K\"ahler. This is very similar to Peter Gilkey's result that if two complex manifolds are isospectral, they are either both K\"ahler or neither is ~\cite{PG}.

However, we turn our attention away from the cases in which $\eta$ is zero and try to understand it for general Hermitian manifolds. From the Theorem 1, if we have control over $\xi$ and $\nabla \xi$ (where $\nabla$ is with respect to the Levi-Civita connection), then we will be able to obtain bounds on the spectrum of $\square$. By equation 9, this reduces to finding estimates on  $\eta$ and $\nabla \eta$. We now derive such estimates in terms of the curvature of the Levi-Civita and Chern connections.

\subsection{Structural Inequalities on Hermitian Manifolds}

From a geometric (and non-rigorous) point of view, the only metric invariants should come from curvature so the need for twisting of a unitary frame can be entirely determined by how ``non-flat" the underlying space is. Furthermore, the deformation of shapes and angles occurs due to the presence of curvature, not its derivatives, so we expect the torsion is bounded somehow by the curvature and how the curvature of the Chern and Levi-Civita connections differ. This phenomena has been noted before and studied in a somewhat different context in ~\cite{SMS}, which shows that if a complex structure exists at all, certain parts of the Weyl tensor must vanish.
 These observations suggest that torsion should be at most a ``zero-th" order phenomena with respect to the curvature and we seek to formalize this intuition. The following relies heavily on the results from ~\cite{YZ}, which we cite repeatedly.

By equation 41 of ~\cite{YZ}, we have that given any type $(1,0)$ vector X,

$$R^h_{X \bar X X \bar X} -R_{X \bar X X \bar X} = \sum_k |T_{kX}^X|^2$$

Therefore, in any unitary frame $\{ e_i\}$,
\begin{eqnarray*} 
 \sum_{i=1}^n R^h_{i \bar  i i \bar i} -R_{i \bar i i \bar i}  &= &  \sum_{i=1}^n \sum_k |T_{ki}^i|^2 \\
										 &= &  \sum_{k=1}^n \sum_i |T_{ki}^i|^2 \\
										 & \geq & \frac{1}{(n-1)} \sum_{k=1}^n |\eta_k|^2 \\
										  & = & \frac{1}{(n-1)} ||\eta||_2^2 \\
\end{eqnarray*} 

However, we can get stronger estimates using the structure theorems of ~\cite{YZ}.

For clarity, we write the results of Lemma 7 of ~\cite{YZ} here.

\begin{theorem*}{(Lemma 7)}
Let $(M^n,g)$ be a Hermitian manifold and let $p \in M$. Let $\{e_i\}$ be a unitary frame near $p$ such that $\theta |_p = 0$. Then, at the point $p$ we have:
$$2T^k_{ij}{}_{,\bar l} = R^h_{j \bar l i \bar k} -R^h_{i \bar l j \bar k}$$
$$ T^l_{ij}{}_{, k} =  R_{i j k \bar l} - T_{ri}^l T_{jk}^r + T_{rj}^l T_{ik}^r$$
$$2R_{ij \overline{kl}} =  T^l_{ij}{}_{, \bar k} - T^k_{ij}{}_{,\bar l} + 2T_{ij}^r \overline{T_{kl}^r} + T_{ri}^k \overline{T_{rl}^j}+T_{rj}^l \overline{T_{rk}^i}-T_{ri}^l \overline{T_{rk}^j}-T_{rj}^k \overline{T_{rl}^i}$$
$$R_{k \bar l i \bar j} = R^h_{k \bar l i \bar j} - T^j_{ik}{}_{,\bar l} - \overline{ T^i_{jl}{}_{,\bar k}} +T_{ik}^r \overline{T_{jl}^r} - T_{rk}^j \overline{T_{rl}^i} - T_{ri}^l \overline{T_{rj}^k}$$
where r is summed through and $h_{ , i} = e_i(h)$ and $h_{ ,\bar i} = \bar e_i(h)$.
\end{theorem*}

Now we recall the two norms that measure how much the metric fails to be K\"ahler. We define  $\| R^h-R \|^2_\star$ as:

 $$\| R^h-R \|^2_\star = \sum_{i,j,k,l}| R^h_{i \bar j k \bar l} - R_{i \bar j k \bar l} |^2 + 2\sum_{i,j,k,l}| R_{i j \bar k \bar l} |^2 $$

Recall that $R^h_{XY \bar Z \bar W} = 0$ by Gray's theorem so this can be thought of as a norm of the differences of Riemannian and Hermitian curvatures.

Also, we define $\| R^h-R \|^2_{\star \star}$ in the following way:

 $$\| R^h-R \|^2_{\star \star} = \sum_{i,j,k,l}| R_{i j k \bar l} |^2 $$

Since $R^h_{XYZ \bar W} = 0$, the above notation is meaningful. The next theorem shows how $|R^h-R|_\star$ measures how much a metric fails to be K\"ahler. Recall that one possible definition of the K\"ahler condition is that the torsion identically vanishes.

\begin{theorem*}
The following inequalities hold pointwise:
\\
$||T||^2 \leq  |R^h-R|_\star$ and $||\eta||^2 \leq  |R^h-R|_\star$
\end{theorem*}

\begin{proof}
From equation 41 of ~\cite{YZ},
  \begin{eqnarray*} 
   \frac12(R^h_{X \bar X Y \bar Y} + R^h_{Y \bar Y X \bar X}) - R_{X \bar Y Y \bar X} & = &  \sum_{k} ( |T_{XY}^k|^2 + 2Re(T_{kY}^Y \overline{T_{kX}^X}))
\end{eqnarray*}
We also have that $R_{XY \bar X \bar Y} = R_{X \bar X Y \bar Y}- R_{X \bar Y Y \bar X}$. Therefore, we can rewrite the left-hand side of the above equation:

  \begin{eqnarray*} 
 & &  \sum_{k} ( |T_{XY}^k|^2 + 2Re(T_{kY}^Y \overline{T_{kX}^X}))\\
    & = & \frac12(R^h_{X \bar X Y \bar Y} + R^h_{Y \bar Y X \bar X}) - R_{X \bar X Y \bar Y} - R_{XY \bar X \bar Y} \\
   & = & \frac12 ( R^h_{X \bar X Y \bar Y} - R_{X \bar X Y \bar Y}) +  \frac12 ( R^h_{Y \bar Y X \bar X} - R_{Y \bar Y X \bar X}) - R_{XY \bar X \bar Y} \\
\end{eqnarray*}

We want to gain a better understanding of $\sum_{k=1}^n 2 Re( T_{kX}^X \overline{T_{kY}^Y})$. Choose a unitary frame and let $\sum_{i} T^i_{ki} = \eta_k$.

   \begin{eqnarray*}  \sum_{i} \sum_{j}   Re(T^i_{ki} \overline{ T^j_{kj} } )
   & = & Re( \sum_{i} T^i_{ki}  (\sum_{j}  \overline{ T^j_{kj}}  ) \\
 & = & Re( \sum_{i} T^i_{ki} ( \bar \eta_k)) \\
 & = & Re( \bar \eta_k \sum_{i} T^i_{ki}) \\
 & = & Re( \bar \eta_k \eta_k) \\
 & = & |\eta_k|^2
 \end{eqnarray*}

 That is to say, $$\sum_{i} \sum_{j} 2 Re(T^i_{ki} \overline{ T^j_{kj} } ) = 2|\sum_{i} T^i_{ki}|^2 = 2 |\eta_k|^2$$
 
 Thus, we have the following:
   \begin{eqnarray*} 
& & \sum_i \sum_{j} \frac12 ( R^h_{i \bar i j \bar j} - R_{i \bar i j \bar j}) +  \frac12 ( R^h_{j \bar j i \bar i} - R_{j \bar j i \bar i}) - R_{i j \bar i \bar j} \\
& = & \sum_i \sum_{ j}  \sum_{k} ( |T_{ij}^k|^2 + 2Re(T_{ki}^i \overline{T_{kj}^j})) \\
& = & \sum_{k}  \sum_i \sum_{ j} ( |T_{ij}^k|^2 + 2Re(T_{ki}^i \overline{T_{kj}^j}))   \\
& = & \sum_{k}  \left( (\sum_i \sum_{ j}  |T_{ij}^k|^2) + 2|\eta_k|^2 \right)  \\
& = & ||T||^2 + 2 ||\eta||^2
\end{eqnarray*}
 
This then immediately proves that $||T||^2 \leq  \|R^h-R\|_\star$. 
\end{proof}
There are many terms in $\| R^h - R \|_\star$ that are not needed to control the torsion, but we define $|R^h-R|_\star$ as such in view of the next lemma.

\begin{lemma}

The following inequality holds where $\nabla^{c''}$ is the derivative with respect to the Chern connection of a $(0,1)$ vector:
$\| \nabla^{c''} T \| \leq   \| R^h-R \|_\star$
\end{lemma}

\begin{proof}
We use the first equation of Lemma 7 and the following Bianchi identity:
 \begin{eqnarray*}
R_{i \bar j k \bar l} - R_{k \bar j i \bar l} &=& R_{i \bar j k \bar l}  + R_{\bar j k  i \bar l} \\
& = & - R_{k i \bar j \bar l} =  R_{i k \bar j \bar l}
\end{eqnarray*}
Combining these we find that:
$$e_{\bar l}( T_{ik}^j )= (R^h_{i \bar j k \bar l} -  R_{i \bar j k \bar l}) -  (R^h_{k \bar j i \bar l} - R_{k \bar j i \bar l}) - R_{i k \bar j \bar l} $$
\end{proof}

We can also bound the derivative with respect to a $(1,0)$ vector.
\begin{lemma}
The following inequality holds where $\nabla^{c'}$ is the derivative with respect to the Chern connection of a $(1,0)$ vector:
$$| \nabla^{c'} T| \leq C(n) \| R^h-R \|_\star + \| R^h-R \|_{\star \star}$$
\end{lemma}

This is an immediate consequence of the equation from Theorem 4 and the following equation from ~\cite{YZ}:
$$ T^l_{ij}{}_{, k} =  R_{i j k \bar l} + T_{rj}^l T_{ik}^r - T_{ri}^l T_{jk}^r$$

Using a straightforward computation, we can relate the derivatives of the torsion tensor with respect to the Levi-Civita connection to the derivatives of the torsion tensor with respect to the Chern connection and a quadratic expression in torsion. Therefore, we also have the following result:

\begin{theorem*}
Let $T$ be the torsion tensor and $\nabla T$ the derivative of the torsion tensor with respect to the Levi-Civita connection. Then there exists a  constant $C^\prime(n)$ that grows at most linearly in $n$ such that the following inequality holds:
$$||\nabla T|| \leq   C^\prime(n) \|R^h-R \|_\star + \| R^h -R \|_{\star \star} $$
\end{theorem*}

As a consequence to this result, we have the following observation.

\begin{theorem*}
Let $(M^{2n}, g, J)$ be a compact complex manifold with pluriclosed metric g.  Consider the following initial value problem
$$ \frac{\partial}{\partial t} \omega = \partial \partial^{*} \omega + \overline{ \partial \partial^{*}} \omega + \frac{\sqrt{-1}}{2} \partial \bar \partial \log \det g$$
$$w(0)= g(\cdot, J \cdot)$$
There exists a constant $c(n)$ depending only on $n$ such that there exists a unique solution $g(t)$ for
$$t \in [0, \frac{c(n)}{\max(| R^h |_{C^0(g_0)}, | R |_{C^0(g_0)})}]$$
Furthermore, we have existence until either $| R^h |_{C^0(g_0)}$ or  $| R |_{C^0(g_0)}$ blow up.
\end{theorem*}

This follows immediately from Theorem 1.2 in ~\cite{ST}. This shows that pluriclosed flow exists until either the Chern curvature or Riemannian curvature blows up. In the first case, $M$ fails to be a well behaved complex manifold and in the latter, $M$ fails to be a well behaved Riemannian manifold. It should be noted that this observation is weaker than the improved regularity theorem of ~\cite{ST2}, which provides a Bismut-Ricci curvature term which suffices to prove regularity for pluriclosed flow.

In future preprints, we will explore these inequalities in much greater depth. One can derive monotonicity results and other structure theorems by leveraging the equations we have used in this paper. However, for the purposes of bounding the principal eigenvalue, what we've done is sufficient.

\section{An estimate on the principal eigenvalue of the Complex Laplacian}

We now combine the two theorems from the previous section with our estimate from the section before.

\begin{theorem*}
Let $(M^n, h)$ be a compact Hermitian manifold without boundary with Riemannian curvature satisfying $Ric \geq -(n-1)k$ for $k \geq 0$ and $diam(M) = d$. Consider the complex Laplacian $\square$. Let $u$ satisfy $\square u = \lambda u$ and $K = n^2 (k+\| R-R^h \|_\star + \| R-R^h \|_{\star \star} )d^2$. Then there exists an uniform $C >0$ such that the following estimate holds:
$$\lambda \geq \frac{1}{4nd^2}  \frac{\left(1 + \sqrt{1+4CK}\right)^2-CK }{\exp \left( 1 + \sqrt{1+4 CK } \right) }$$
\end{theorem*}

We can simplify this estimate. There exists an uniform $C >0$ such that:
$$\lambda \geq \frac{1}{4n}  \frac{\left( \frac{2}{d^2} + 3Cn^2 (k+\| R-R^h \|_\star +  \| R-R^h \|_{\star \star} ) \right)}{\exp \left( 1 + \sqrt{1+4 CK } \right) }$$

This is still complicated but it only involves the dimension $n$, the diameter $d$, a lower bound on the Ricci curvature $k$, and a measure of how much the metric fails to be K\"ahler involving only curvature. From the point of view of this estimate, the last term acts in the same way as the Ricci curvature term. 

In a future preprint, we will explore monotonicity formulae to show that in some ways the Hermitian curvature dominates the Riemannian curvature. For instance, it is well known that the Hermitian scalar curvature dominates the Riemmanian scalar curvature. However, our monotonicity results are not yet strong enough to bound the spectrum solely using the Riemannian curvature. Nonetheless, this may be possible in light of a theorem due to Yang and Zheng, which states that if the Riemannian curvature is Gray-K\"ahler-like and the space is compact, then the metric is balanced. Since balanced metrics satisfy $2 \square = \Delta$ on functions, the two operators have the same spectrum. This provides a condition which only depends on the Riemannian curvature and the diameter, that ensures that spectral geometry of the complex laplacian is the same as that of the Laplace-Beltrame operator. However, we cannot hope to control the entire torsion tensor solely using the Riemannian curvature because there exists Riemannian flat metrics that do not have vanishing torsion (such as the tori in \cite{BSV}). Also, the example in ~\cite{YZ} of a non-compact Gray-K\"ahler-like surface that is not K\"ahler shows that Gray-K\"ahler-like alone does not imply balanced.

\begin{conjecture*}
Given a compact complex manifold $(M^n, h)$, there exists $C$ depending only on the Riemannian geometry such that if $\square u = \lambda u$, then $\lambda \geq C$.
\end{conjecture*}

Using Theorem 1, this conjecture would be a consequence of the following conjecture.

\begin{conjecture*}
Given a compact complex manifold $(M^n, h)$, there exists $C$ depending only on the Riemannian geometry such that the $C^1$ norm of $\| \eta \| \leq C$
\end{conjecture*}

We suspect that the $C$ in the previous two conjectures can be expressed in terms of the dimension, the diameter and injectivity radius of $M$, an upper and lower bound of the Ricci curvature, the Riemannian curvature tensor, and a $C^1$ bound on the scalar curvature.
These may be more accessible in the conformally balanced or conformally K\"ahler case, in which case the analysis done on the Yamabe problem may be useful. We can also use the Riemannian geometry to bound a weaker norm of $\| \eta \|,$ but have not been able to establish bounds on the full $C^1$-norm. One strategy towards proving Conjecture 6 would be to weaken the norm on $\eta$ required to establish bounds on the eigenvalue. 

We use this result to reprove a result originally due to Gauduchon~\cite{PGa}.

\begin{theorem*}
Given a compact complex manifold $(M^n, h)$ and let $\| Riem \|$ be the pointwise norm of the Riemannian curvature tensor. Then the torsion one-form $\eta$ satisfies the following inequality:
$$ \| \eta \|^2_{L^2} \leq \frac{n^2}{4} \| Riem \|_{L^2}$$
\end{theorem*}

\begin{proof}

Using the third formula in Lemma 7, if we set $i=k$ and $j=l$ and sum up over both we obtain the following:

$$2R_{ij \overline{ij}} =  T^j_{ij}{}_{, \bar i} - T^i_{ij}{}_{,\bar j} + 2T_{ij}^r \overline{T_{ij}^r} + T_{ri}^i \overline{T_{rj}^j}+T_{rj}^j \overline{T_{ri}^i}-T_{ri}^j \overline{T_{ri}^j}-T_{rj}^i \overline{T_{rj}^i}$$

This simplifies to
$$ R_{ij \overline{ij}} =  \eta_{i}{}_{, \bar i} + \eta_i \bar \eta_i $$
 where we sum over repeated indices.

We consider the following expression
\begin{eqnarray*}
\bar \partial \partial \omega^{n-1} & = & \bar \partial (-2 \eta \wedge \omega^{n-1})\\
& = & -2 \bar \partial \eta \wedge \omega^{n-1} - 4 \eta \wedge \bar \eta \wedge \omega^{n-1} \\
\end{eqnarray*}
 
Combining these two formulas, we obtain:

\begin{eqnarray*}
\bar \partial \partial \omega^{n-1} & = &  -2 (R_{ij \overline{ij}} - |\eta|^2) \omega^{n} - 4 \eta \wedge \bar \eta \wedge \omega^{n-1} \\
& = &  -2 R_{ij \overline{ij}} \omega^{n} - 2 |\eta|^2  \omega^{n} \\
\end{eqnarray*}

Integrating this formula, the left hand side is zero whereas the right hand side involves the curvature and the $L^2$ norm of $\eta$. 
Thus we have the following equation.

$$ \| \eta \|^2_{L^2} \leq \| R_{ij \overline{ij}} \|_{L^1}$$

Without knowing the complex structure, there is no way to take the sum on the right hand side. However, we can estimate the term.
$$|R_{ij \overline{ij}}|^2 \leq \frac{n^2}{4} \| Riem \|^2$$

Thus, we obtain the desired estimate. This is exactly equivalent to Gauduchon's result, so can be strengthened. Gauduchon's result states that 
$$ \delta \theta + |\theta|^2 = \frac{2n-2}{2n-1} S - 2 \langle W(\omega), \omega \rangle$$ Here, $\theta = -2 (\eta + \bar \eta)$, $\delta$ is the codifferential, $S$ is the scalar curvature, $W$ is the Weyl tensor if viewed as an endomorphism of $2$-forms and $\omega$ is the K\"ahler form. This immediately implies:

$$ \| \eta \|^2_{L^2} \leq \int_M \frac{1}{4}( \frac{2n-2}{2n-1} S + 2 |W|) $$

% \textit{ We can improve this estimate. \\ A straightforward manipulation shows that  $R_{ij \overline{ij}} = 2 R(i,j,Ji,Jj) - 2R(i,Jj, Ji,j)$. The term is zero unless $i \neq j$ and since $J$ is orthogonal to the metric $g(i,Ji) =g(j,Jj)=0$. Therefore, the first term only involves the Weyl curvature unless  $Ji = j$, in which case we have  $2R(i,Ji, Ji, -i) = 2R(i,Ji, i, Ji)$. If $i \neq j$, then the last term is either 0 or completely contained in the Weyl tensor. Therefore, the last term involves terms with the Weyl tensor as well as some that involve the sectional curvature $R(i,Ji,i,Ji)$. This implies that if a conformally flat metric has non-positive sectional curvature, then the only complex structures it admits are balanced. Can we do better? Is there a natural curvature condition that forces $\sum_i R(i,Ji,i,Ji) \leq 0$ for all orthonormal frames and choices of $J$? The result of Hernandez Lamoneda seems to imply that the entire sum is negative if the metric is quarter-pinched. Can we use the result of Gray to rewrite the sum as 4R(i,j,Ji,Jj) to make our lives easier?}

\end{proof}

\section{Some special cases}

We are able to prove Conjecture 5 in several special cases.

\subsection{Globally conformally flat metrics}
The first case is where $(M^n, g)$ is globally conformally flat.

\begin{theorem}
Let $(M^n, g)$ be a compact globally conformally flat Hermitian manifold. Let $K= \inf_{x \in M} Ric~ M$, $k= \sup_{x \in M} Ric~ M$ $R$ be the scalar curvature of $M$, $d$ be the diameter of $M$ and $i$ be the injectivity radius of $M$. If $\lambda_1$ is the principle eigenvalue of the complex Laplacian $\square$, then we have the following estimate:
$$\lambda \geq C(d, K, k, n, | \nabla R|, R^2, i)$$
\end{theorem}

In order to do this, we note that $\square = \frac12 \Delta + \xi( \nabla)$ where $\xi$ is the Lie form. Furthermore, we note that if $(M^n, g)$ is conformally flat ( where $e^{2f} g$ is flat), then it is also conformally balanced. As we observed that the Lie form is gradient if a metric is conformally balanced, this means that the complex Laplacian is actually a Witten-Laplacian (modulo a factor of $2$) with drift $\nabla e^{-2f}$. In this case, if one has $C^2$ bounds on $e^{-2f}$, one can often obtain lower bounds on the spectrum. However, we will use the bound from this preprint instead of the standard result using the Bakry-Ricci curvature. The reason to do this is that the bound using Bakry-Ricci curvature can become trivial if the curvature is too negative. Our bound is often much weaker when both apply, but does not display this type of threshold behavior. Since we will need to use a Harnack type result to estimate the $C^0$ bounds on $f$ on the manifold and then the conformal structure to obtain $C^2$ bounds, we obtain weak bounds that interact poorly with cutoff thresholds. 

We start by proving a Harnack inequality for $(M^n, g)$ for the linear equation $\Delta u = fu$. This proof is completely standard ~\cite{SY} and we provide it in detail only to note explicitly what quantities are used for the estimate.

\subsection{A Harnack inequality}

\begin{lemma}
Suppose $u$ is a positive function satisfying $\Delta u = fu$ on $(M^n, g)$. Then $\inf u \leq C_1 \sup u$ where $C_1$ depends only on $Ric
~ M, n, diam, i, f^2,$ and $|\nabla f|$.
\end{lemma}

Pick an orthonormal frame at a point which satisfies the same conditions as the one before.

\begin{eqnarray*} 
\frac12 \Delta( |\nabla u|^2) &=& \sum u_{ij}^2 + \sum u_i (\Delta u)_i + Ric (\nabla u, \nabla u) \\
&= &  \sum u_{ij}^2 + \sum u_i (f u)_i + Ric (\nabla u, \nabla u) \\
&\geq &  \sum u_{ij}^2  + \sum f |u_i|^2  + \sum u_i u f_i - K |\nabla u|^2 \\
&\geq &  \sum u_{ij}^2  - \frac12(|\nabla u|^2 + u^2)  |\nabla f| - (K-f) |\nabla u|^2 \\
\end{eqnarray*}

Set $A = (K- f + \frac12 |\nabla f| )$. Then we have 
\begin{eqnarray*} 
\frac12 \Delta( |\nabla u|^2) &\geq& \sum u_{ij}^2 - A |\nabla u|^2 - \frac12 |\nabla f|  u^2 
\end{eqnarray*}

We pick normal coordinates at $x_0$ so that $u_i = 0$ for $i> 1$ and $u_1 = |\nabla u|$
$$\nabla_j (|\nabla u|) = \nabla_j \sqrt{\sum u_i^2} = \frac{ \sum u_i u_{ij}}{|\nabla u|} = u_{1j}$$

So,
$$|\nabla (|\nabla u|)|^2 = \sum u_{1j}^2$$

$$|\Delta (|\nabla u|)^2| =2 |\nabla u| \Delta (|\nabla u|)+ 2 |\nabla (|\nabla u|)|^2 $$

\begin{eqnarray*} 
 |\nabla u| \Delta (|\nabla u|)+ \sum u_{1j}^2 &\geq& \sum u_{ij}^2 - A |\nabla u|^2 - \frac12 |\nabla f|  u^2 
\end{eqnarray*}

This implies
\begin{eqnarray*} 
 |\nabla u| \Delta (|\nabla u|)+ \sum u_{1j}^2 &\geq& \sum u_{ij}^2 - A |\nabla u|^2 - \frac12  |\nabla f|  u^2 
\end{eqnarray*}

Thus,
\begin{eqnarray*} 
 |\nabla u| \Delta (|\nabla u|) +  A |\nabla u|^2 +  \frac12 |\nabla f|  u^2 &\geq& \sum u_{ij}^2 - \sum u_{1j}^2 \\
&\geq& \sum_{i \neq 1} u_{i1}^2 + \sum_{i \neq 1} u_{ii}^2 \\
&\geq& \sum_{i \neq 1} u_{i1}^2 + \frac{1}{n-1} \left( \sum_{i \neq 1} u_{ii} \right)^2 \\
\end{eqnarray*}

$$\Delta u = \sum u_{ii} = fu$$

Thus, $f u - u_{11} = \sum u_{ii}$

\begin{eqnarray*} 
 ( \sum_{i \neq 1} u_{ii})^2 &=& (f u - u_{11})^2 \\
&\geq& \frac12 u^2_{11}  - f^2 u^2\\
\end{eqnarray*}

Thus, 
\begin{eqnarray*} 
 |\nabla u| \Delta (|\nabla u|) +  A |\nabla u|^2 +  (\frac12 +f^2)u^2 &\geq& \frac{1}{2(n-1)} |\nabla (|\nabla u|)|^2 \\
\end{eqnarray*}

Set $B = (\frac12 |\nabla f| +f^2)$

Then, 
\begin{eqnarray*} 
 |\nabla u| \Delta (|\nabla u|) +  A |\nabla u|^2 +  Bu^2 &\geq& \frac{1}{2(n-1)} |\nabla (|\nabla u|)|^2 \\
\end{eqnarray*}

Then let $\phi = \dfrac{|\nabla u|}{u}$

$$\nabla \phi = \frac{\nabla | \nabla u|}{u} - \frac{\nabla u | \nabla u|}{u^2}$$

$$ | \nabla u| = \phi u$$

\begin{eqnarray*} 
 \Delta (|\nabla u|) &=& u \Delta \phi + \phi \Delta u + 2 \nabla \phi \cdot \nabla u
\end{eqnarray*}

\begin{eqnarray*} 
 \Delta (\phi u) &=& u \Delta \phi + \phi \Delta u + 2 \nabla \phi \cdot \nabla u \\
&=& u \Delta \phi + \phi f u + 2 \nabla \phi \cdot \nabla u \\
&=& u \Delta \phi + |\nabla u| f + 2 \nabla \phi \cdot \nabla u \\
\end{eqnarray*}

Thus, 

\begin{eqnarray*} 
 \Delta (\phi) &=& \frac{ \Delta (|\nabla u|)}{u} - \frac{2 \nabla \phi \cdot \nabla u}{u} - f \phi \\
&=& \frac{ |\nabla u| \Delta (|\nabla u|)}{|\nabla u| u} - \frac{2 \nabla \phi \cdot \nabla u}{u} - f \phi \\
&\geq& \frac{1}{|\nabla u| u} \left( \frac{1}{2(n-1)} |\nabla (|\nabla u|)|^2 - A |\nabla u|^2 -  Bu^2   \right) - \frac{2 \nabla \phi \cdot \nabla u}{u} - f \phi \\
&=& \frac{1}{|\nabla u| u} \frac{1}{2(n-1)} |\nabla (|\nabla u|)|^2 - A \phi -  \frac{B}{\phi}   - \frac{2 \nabla \phi \cdot \nabla u}{u} - f \phi \\
\end{eqnarray*}
Recalling that $A = (K- f + \frac12)$, set $\alpha = (K+ \frac12)$ to get
\begin{eqnarray*} 
 \Delta (\phi) &\geq& \frac{1}{|\nabla u| u} \frac{1}{2(n-1)} |\nabla (|\nabla u|)|^2 - \alpha \phi -  \frac{B}{\phi}   - \frac{2 \nabla \phi \cdot \nabla u}{u}  \\
\end{eqnarray*}

Let $\epsilon = \frac{1}{n-1} >0$

\begin{eqnarray*} 
\frac{2 \nabla \phi \cdot \nabla u}{u} &=& (2-\epsilon)\frac{\nabla \phi \cdot \nabla u}{u}+\frac{\epsilon \nabla \phi \cdot \nabla u}{u} \\
&=& (2-\epsilon)\frac{\nabla \phi \cdot \nabla u}{u}+ \epsilon \frac{ |\nabla (|\nabla u|)| | \nabla u|}{u} -  \epsilon \frac{ | \nabla u|^3}{u^3} \\
\end{eqnarray*}

Then  
\begin{eqnarray*} 
\epsilon \frac{ |\nabla (|\nabla u|)| | \nabla u|}{u} & = & \epsilon \frac{ |\nabla (|\nabla u|)| }{ (| \nabla u| u)^{1/2}} \frac{ | \nabla u|^{3/2}}{u^{3/2}} \\
& \leq & \frac{\epsilon}{2}\left( \frac{ |\nabla (|\nabla u|)|^2 }{ (| \nabla u| u)} + \frac{ | \nabla u|^{3}}{u^{3}} \right)\\
& = & \frac{1}{2(n-1)}\left( \frac{ |\nabla (|\nabla u|)|^2 }{ (| \nabla u| u)} +\phi^3 \right)\\
\end{eqnarray*}

Therefore,
\begin{eqnarray*} 
\Delta (\phi)& \geq & -\alpha \phi - \frac{B}{\phi}+ \frac{1}{2(n-1)} \phi^3 +  (2-\epsilon)\frac{\nabla \phi \cdot \nabla u}{u} \\
\end{eqnarray*}

Consider the maximum point of $\phi$. At this point, $\nabla \phi =0$ and  $\Delta (\phi) \leq 0$ and so multiplying through by $\phi$, we obtain

\begin{eqnarray*} 
0& \geq & -\alpha \phi^2 - B+ \frac{1}{2(n-1)} \phi^4 +  (2-\epsilon)\frac{\nabla \phi \cdot \nabla u \phi}{u} \\
0& = & -\alpha \phi^2 - B+ \frac{1}{2(n-1)} \phi^4 \\
\end{eqnarray*}

Then, we find that
$$\phi^2 \leq (n-1)\left( B+\sqrt{B^2 + 2\frac{ \alpha}{n-1}} \right) $$

This is what we need to bound in order to get a Harnack inequality. We can integrate this along the geodesics to get upper and lower bounds on $\log u$ using the diameter of $M$.

\subsection{Conformally Flat Geometry}

 If we set $f =R$, this solves the Yamabe problem in the flat case. If we force the volume to be preserved, we can integrate the result from the lemma to get upper and lower bounds on $u$ using the diameter (since we know that $\sup u > 1$ and $\inf u < 1$ in that case). More directly, this gives us bounds on $\nabla \log u$.

Now, let $f = \frac{2}{n-2} \log u$.

Then the conformal formula for Ricci curvature shows us that

$$\tilde R_{ij} = R_{ij} - (n-2)\left[ \nabla_i\partial_j f - (\partial_i f)(\partial_j f) \right] + \left( \triangle f - (n-2)\|\nabla f\|^2 \right)g_{ij} $$
$$\tilde R  = e^{-2f}\left(R + 2(n-1)\triangle f - (n-2)(n-1)\|\nabla f\|^2\right) $$

Since $\tilde g$ is flat, we have

\begin{eqnarray*}
0 & = & \tilde R_{ij}  - \frac{\tilde R}{2(n-1)}\tilde g_{ij} \\
& = &  R_{ij} - (n-2)\left[ \nabla_i\partial_j f - (\partial_i f)(\partial_j f) \right] + \left( \triangle f - (n-2)\|\nabla f\|^2 \right)g_{ij} \\
&   &  - \left( \frac{R}{2(n-1)} + \triangle f - \frac{(n-2)}{2}\|\nabla f\|^2\right) g_{ij} \\
& = &  R_{ij} - \frac{R}{2(n-1)}  g_{ij} - (n-2)\left[ \nabla_i\partial_j f - (\partial_i f)(\partial_j f) \right] - (n-2)\|\nabla f\|^2g_{ij} + \frac{(n-2)}{2}\|\nabla f\|^2g_{ij}
\end{eqnarray*}

Thus we have 

\begin{eqnarray*}
\|- (n-2) \nabla_i \partial_j f \|  &= & \|R_{ij} - \frac{R}{2(n-1)}  g_{ij} + (n-2)(\partial_i f)(\partial_j f)  - \frac{(n-2)}{2}\|\nabla f\|^2g_{ij} \| \\
&\leq & \|R_{ij} - \frac{R}{2(n-1)}  g_{ij} \| + \| (n-2)(\partial_i f)(\partial_j f)  - \frac{(n-2)}{2}\|\nabla f\|^2g_{ij} \| \\
\end{eqnarray*}

Setting $i=j$ (this is what we need in the proof of the lower bound on $\lambda$), and working in normal coordinates, we have

\begin{eqnarray*}
\|- (n-2) \nabla_i \partial_i f \| &\leq & \|R_{ii} - \frac{R}{2(n-1)} \| + \| (n-2)(\partial_i f)^2 \|  + \| \frac{(n-2)}{2}\|\nabla f\|^2\| \\
&\leq & \|R_{ii} - \frac{R}{2(n-1)} \|  + \| \frac{3(n-2)}{2}\|\nabla f\|^2\|
\end{eqnarray*}

The Harnack estimate provides bounds on the last term. Then the bounds on the Ricci and scalar curvature give us bounds on the first term on the right hand side. Now we have the needed $C^2$ bounds on $\log u$. However, in order to get $C^2$ bounds on $e^{-2f}$, we note that 

$$ \nabla_i e^{-2f} =  -2 (\nabla_i f) e^{-2f} $$
$$ \nabla_i \nabla_i e^{-2f} = 4 (\nabla_i f)^2 e^{-2f} - 2 (\nabla_i \nabla_i f) e^{-2f} $$

Thus, with bounds on $|\nabla_X \nabla_X f|$ , $|\nabla_X f|$, and upper and lower bounds on $e^{-2f}$, we obtain  $C^2$ bounds on $e^{-2f}$ in terms of $d, K, k, n, |\nabla R|$, and $R^2$. We can feed this into the theorem from the preprint to get lower bounds on $\lambda$ where $\lambda$ is the principal eigenvalue of the complex Laplacian.

Note that from the perspective of differential geometry and differential topology, globally conformally flat metrics are well understood. However, such metrics admit a rich moduli of orthogonal complex structures ~\cite{KYZ}, many of which are non-K\"ahler. Without such a result, these metrics form a large obstruction to Conjecture 5. This also illustrates one possible approach to proving the conjecture in greater generality. One would start by proving that all complex structures orthogonal to a particular metric in a conformal class satisfy the estimate and then study how the Ricci curvature is affected by conformal transformations of the metric.

The above argument proves the conformal stability part of the result in the conformally Ricci-flat case. This shows is that any estimate on the torsion one-form is conformally stable, in that if one can obtain a $C^1$ estimate on the torsion one-form of a Ricci-flat metric, then one can obtain a $C^1$ estimate on the torsion one form for any conformal deformation of that metric. For non-flat metrics, we only have an $L^2$ estimate on the torsion one-form for all complex structure so this is not yet strong enough to finish the proof.

\subsection{$k$-Gauduchon metrics}

We are able to establish lower bounds on the spectrum in one other case. The bound is given by an argument by contradiction, so is not effective, but we will present a numeric lower bound in a future preprint.

Let $G$ be the Gauduchon curvature. That is 
$$G = \frac{2n-2}{2n-1}S - \langle W(\omega), \omega \rangle$$

Consider 
$$\alpha_k = i \partial \bar \partial (\omega^k) \wedge \omega^{n-k-1}$$

Using a slight generalization of the calculation in ~\cite{MT}, we see that

$$\alpha_k =  \frac{2k \omega^n}{n(n-1)(n-2)}  \left( -(n-2) \eta_i,_{\bar i} + (2k-1)|\eta|^2 + (n-k-1)|T|^2 \right) $$

Substituting in for the Gauduchon curvature, we obtain:
$$\alpha_k =  \frac{2k \omega^n}{n(n-1)(n-2)} \left( -(n-2)G + (2k-n)|\eta|^2 + (n-k-1)|T|^2 \right) $$

If the metric is $k$-Gauduchon, then $0 = i \partial \bar \partial (\omega^k) \wedge \omega^{n-k-1}$. For $k>\frac{n}{2}$, then 
$$(n-2)G = (2k-n) |\eta|^2 + (n-k-1)|T|^2$$
Since $(2k-n), (n-k-1) >0$, this gives us pointwise bounds on torsion in terms of $G$ alone.

\section{A lower bound on the spectrum}
 \begin{theorem}Let $(M^{2n}, g)$ be a compact Riemannian manifold and $J$ be an orthogonal complex structure which is $k$-Gauduchon for some $k>\frac{n}{2}$. Then the spectrum of the complex Laplacian is bounded below by some constant $C$ depending only on $(M^{2n}, g)$, independent of $J$.
 \end{theorem}
 
 The complex Laplacian on a function can be written as $\frac12 \Delta u + (2\eta +2 \bar \eta)(\nabla u)$
 
 Suppose there were a sequence of complex structures such that
 $\frac12 \Delta u_i + (2\eta_i +2 \bar \eta_i)(\nabla u_i) = \lambda_i u_i$ with $\lambda_i \to 0$ where $\sup |u_i| = 1$ for all $ i$ and $\int u_i~ dVol=0$.
 
Since $\eta_i$ is bounded in $L^\infty$, it converges in weak $L^p$ to $\eta_\infty$, after possibly passing to a subsequence. Similarly, since $u_i$ is uniformly bounded in $H^{2,p}$ for large $p$, it converges weakly to $u_\infty$ after passing to a further subsequence. Therefore, $\frac12 \Delta u_i$ converges in weak $L^2$ to $\frac12 \Delta u_\infty$. Furthermore, $\nabla u_i$ converges to $\nabla u_\infty$ strongly and $\lambda_i u_i$ converges to 0 uniformly.
 
 Thus we have,
 $\frac12 \Delta u_\infty + (2\eta_\infty +2 \bar \eta_\infty)(\nabla u_\infty) = 0$ for $u_\infty$ non-constant (since $\sup |u_\infty| = 1$ and $\int u_\infty~ dVol=0$). This contradicts the strong maximum principle.
 
 Since every complex structure is Gauduchon for one metric in the conformal class, if one were able to obtain a $C^2$ estimate on the $\log$ of the conformal factor that makes the K\"ahler form Gauduchon in terms of the Riemannian geometry, this would prove Conjecture 5 in full generality.
 
 This result does not give an explicit lower bound in terms of the geometry. In order to do that, one would need to prove a version of the Faber-Krahn inequality with drift. Such a result is known in flat Euclidean space ~\cite{HNR}, and we will prove a numeric result in a following paper.

\end{document}